\documentclass[11pt]{amsart}
\usepackage{graphicx}

\usepackage{amsmath}
\usepackage{amsthm}
\usepackage{amssymb}
\usepackage{enumerate}

\usepackage{cleveref}

\def\C{\mathcal C}
\def\L{\mathcal L}

\newtheorem{lem}{Lemma}[section]

\newtheorem{prop}[lem]{Proposition}
\newtheorem{coro}[lem]{Corollary}
\newtheorem{thm}[lem]{Theorem}

\newtheorem{de}[lem]{Definition}

\newtheorem{pozn}[lem]{Remark}
\newtheorem{ex}[lem]{Example}

\newtheorem{obs}[lem]{Observation}
\def\pf{\begin{proof}}
\def\pfk{\end{proof}}

\begin{document}
\title{Factor frequencies in generalized Thue-Morse words}
\maketitle
\begin{center}
\author{L{\!'}. Balkov\'a\footnote{e-mail: lubomira.balkova@fjfi.cvut.cz}\\[2mm]
{\normalsize Department of Mathematics FNSPE, Czech Technical
University in Prague}\\
{\normalsize Trojanova 13, 120 00 Praha 2, Czech Republic}}
\end{center}

\begin{abstract}
We describe factor frequencies of the generalized Thue-Morse word ${\mathbf t}_{b,m}$ defined for $b \geq 2, \ m\geq 1, \ b,m \in \mathbb N$, as the
fixed point starting in $0$ of the morphism $$\varphi_{b,m}: k \ \to \ k(k+1)\dots(k+b-1),$$
where $k \in \{0,1,\dots, m-1\}$ and where the letters are expressed modulo $m$.
We use the result of Frid~\cite{Fr} and the study of generalized Thue-Morse words by Starosta~\cite{St}.
\end{abstract}
%%%%%%%%%%%%%%%%%%%%%%%%%%%%%%%%%%%%%%%%%%%%%%%%%%%%%%%%%%%%%%%%%%%%%%%%%%%%%%%%%%%%%%%%%%%%%%%%%%%%%%
%%%%%%%%%%%%%%%%%%%%%%%%%%%%%%%%%%%%%%%%%%%%%%%%%%%%%%%%%%%%%%%%%%%%%%%
%%%%%%%%%%%%%%%%%%%%%%%%%%%%%%%%%%%%%%%%%%%%%%%%%%%%%%%%%%%%%%%%%%%%%%%
\section{Introduction}
The generalized Thue-Morse word ${\mathbf t}_{b,m}$ is defined for $b \geq 2, \ m\geq 1, \ b,m \in \mathbb N$, as the
fixed point starting in $0$ of the morphism $$\varphi_{b,m}:  k \ \to \ k(k+1)\dots(k+b-1),$$ where $k \in \{0,1,\dots, m-1\}$ and where the letters are expressed modulo $m$.
Naturally, the class of generalized Thue-Morse words contains the famous Thue-Morse word
${\mathbf t}_{2,2}$ whose factor frequencies have been determined by Dekking~\cite{De}.

Generalized Thue-Morse words belong to the class of circular fixed points of uniform marked primitive morphisms.
For such a~class, Frid~\cite{Fr} has provided an algorithm for the computation of factor frequencies.
We recall her algorithm in Section 1.
The aim of this paper is to describe the set of frequencies of factors in ${\mathbf t}_{b,m}$ of length $n$ for every $n \in \mathbb N$.
The most direct way is to apply Frid's algorithm. However, there is even an easier way thanks to the knowledge of reduced Rauzy graphs (obtained from the description of bispecial factors by Starosta~\cite{St}) and the invariance of the generalized Thue-Morse word under symmetries preserving factor frequencies. In Section 2, we define reduced Rauzy graphs and their relation to factor frequencies. In Section 3, we explain what a~symmetry is and how it preserves factor frequencies.
The main result is presented in Section 4, where we combine Frid's algorithm, reduced Rauzy graphs, and symmetries in order to get factor frequencies of generalized Thue-Morse words.
Recently, an optimal upper bound on the number of factor frequencies in infinite words whose language is invariant under more symmetries
has been derived in~\cite{Ba}. The generalized Thue-Morse word is an example of infinite words for which the upper bound is not attained, as shown in Section 5.

We ask the reader to consult Preliminaries of the paper Generalized Thue-Morse words and palindromic richness by Starosta~\cite{St}
for undefined terms.
%%%%%%%%%%%%%%%%%%%%%%%%%%%%%%%%%%%%%%%%%%%%%%%%%%%%%%%%%%%%%%%%%%%%%%%%%%%%%%%%%%%%%%%%%%%%%%%%%%%%%%
%%%%%%%%%%%%%%%%%%%%%%%%%%%%%%%%%%%%%%%%%%%%%%%%%%%%%%%%%%%%%%%%%%%%%%%

\section{Factor frequencies of fixed points of circular marked uniform morphisms}
If $w$ is a~factor of an infinite word $\mathbf u$ and if the following limit exists
$$\lim_{|v| \to \infty, v \in \L(\mathbf u)}
\frac {\# \{ \mbox{occurrences of $w$ in $v$} \} }{|v|}\,,$$
then it is denoted by $\rho(w)$ and called the {\em frequency} of $w$.
It is known~\cite{Qu} that factor frequencies of fixed points of primitive morphisms exist.
Generalized Thue-Morse words are fixed points of primitive morphisms, therefore we limit
our considerations in the sequel to primitive morphisms.

Let us recall first a~result of Frid~\cite{Fr}, which is useful for the calculation of factor frequencies in
fixed points of primitive morphisms. In order to introduce the result, we need some
further notions. Let $\varphi$ be a~morphism on ${\mathcal A}^*=\{a_1, a_2, \dots, a_m\}^*$. We associate with $\varphi$ the {\em incidence matrix} $M_\varphi$ given by $[M_\varphi]_{ij}=|\varphi(a_j)|_{a_i}$, where $|\varphi(a_j)|_{a_i}$ denotes the number of occurrences of $a_i$ in $\varphi(a_j)$.
The morphism $\varphi$ is called {\em primitive} if there exists $k \in \mathbb N$ satisfying that the power $M_{\varphi}^k$ has all entries strictly positive. As shown in~\cite{Qu}, for fixed points of primitive morphisms,
\begin{itemize}
 \item factor frequencies exist,
 \item it follows from the Perron-Frobenius theorem that the incidence matrix has one dominant eigenvalue $\lambda$ which is larger than the modulus of any other eigenvalue,
 \item the components of the unique eigenvector $(x_1, x_2,\dots, x_m)^T$ corresponding to $\lambda$ normalized so that $\sum_{i=1}^m x_i=1$ coincide with the letter frequencies, i.e., $x_i=\rho(a_i)$ for all $i \in \{1,2,\dots,m\}$.
\end{itemize}
Let $\varphi$ be a~morphism on ${\mathcal A}^*$. We denote $\psi_{ij}:{\mathcal A}^+ \to {\mathcal A}^+$, where $i,j \in \mathbb N$, the mapping that associates with $v \in {\mathcal A}^+$ the word $\psi_{ij}(v)$ obtained from
$\varphi(v)$ by erasing $i$ letters from the left and $j$ letters from the right, where $i+j < |\varphi(v)|$.
We say that a~word $v \in {\mathcal A}^+$ admits an {\em interpretation} $s=(b_0b_1\dots b_m, i,j)$ if $v=\psi_{ij}(b_0b_1\dots b_m)$, where $b_k \in {\mathcal A}$ and $i <|\varphi(b_0)|$ and $j<|\varphi(b_m)|$. The word $a(s)=b_0b_1\dots b_m$ is an {\em ancestor} of $s$. The set of all interpretations of $v$ is denoted $I(v)$.
Now, we can recall the result of Frid for factor frequencies of fixed points of primitive morphisms.
\begin{prop}\label{prop_ancestors}
Let $\varphi$ be a~primitive morphism having a~fixed point $\mathbf u$ and let $\lambda$ be the dominant eigenvalue of the incidence matrix $M_\varphi$. Then for any factor $v \in {\mathcal L}(\mathbf u)$, it holds
\begin{equation}\label{ancestors}
\rho(v)=\frac{1}{\lambda}\sum_{s \in I(v)}\rho(a(s)).
\end{equation}
\end{prop}
For circular fixed points of uniform marked primitive morphisms, the algorithm of Frid~\cite{Fr} provides the possible frequencies of factors of a~given length and for every frequency, the number of factors having that frequency.
In order to describe her algorithm, we have to recall several notions.
A~morphism $\varphi$ defined on the alphabet ${\mathcal A}$ is called {\em uniform} if all images of letters are of the same length $b$, i.e., $|\varphi(a)|=b$ for all $a \in {\mathcal A}$. In the case of a~uniform primitive morphism $\varphi$, the dominant eigenvalue of the incidence matrix $M_\varphi$ is $\lambda=b$. A~morphism is called {\em marked} if every pair of distinct letter images differs both in the first letter and in the last letter.
Let $\mathbf u$ be a~fixed point of a~morphism $\varphi$ defined on $\mathcal A$, then its factor $w$ contains a~{\em synchronization point} $(w_1, w_2)$ if $w=w_1w_2$ and for every $v_1, v_2 \in {\mathcal A}^*$ and for every factor $s$ of $\mathbf u$, there exists factors $s_1,s_2$ of $\mathbf u$ such that
the following implication holds
$$v_1wv_2=\varphi(s) \Rightarrow s=s_1s_2, \ v_1w_1=\varphi(s_1), \ w_2v_2=\varphi(s_2).$$
In other words, a~synchronization point marks a~boundary between letter images in every occurrence of $w$ in $\mathbf u$.
Any factor $w$ of $\mathbf u$ that contains a~synchronization point is called {\em circular}.
We call a~fixed point $\mathbf u$ of a~morphism $\varphi$ {\em circular (with synchronization delay $L$)} if every factor $w$ of length greater than or equal to $L$ is circular. For uniform marked primitive morphisms, Proposition~\ref{prop_ancestors} takes the following form.
\begin{prop}\label{unique_interpretation}
Let $v$ be a~circular word of a~uniform marked primitive morphism $\varphi$ with the letter image length $b$, then there exists a~unique interpretation of $v$. Moreover, if we denote the unique ancestor of $v$ by $w$, then $\rho(v)=\frac{\rho(w)}{b}$.
\end{prop}
We define the {\em structure ordering number} $K$ for fixed points of circular uniform morphisms as the least integer satisfying $b(K-1)+1 \geq L$, where $b$ is the length of letter images and $L$ is the synchronization delay.
The following statements are to be found in~\cite{Fr} as Proposition~4 and Theorem~5.
\begin{prop}
Let $n \geq K$, there exists a~unique triplet of decomposition parameters $(p(n), k(n), \Delta(n))$, where $p(n) \in \mathbb N$,
$k(n) \in \{K, \dots, b(K-1)\}$, and $\Delta(n) \in \{1,\dots, b^{p(n)} \}$, such that
$$n=b^{p(n)}(k(n)-1)+\Delta(n).$$
\end{prop}
The explicit formulae read $p(n)=\left\lceil \log_b \frac{n}{K-1}\right\rceil-1, \ k(n)=\left\lceil \frac{n}{b^{p(n)}}\right\rceil, \ \Delta(n)=n-b^{p(n)}(k(n)-1).$
\begin{thm}\label{explicit}
Let $\mathbf u$ be a~circular fixed point of a~uniform marked primitive morphism $\varphi$.
Denote ${\mathcal L}_n(\mathbf u)=\{v_1^{(n)}, v_2^{(n)}, \dots, v_{{\mathcal C}(n)}^{(n)}\}$.
For all $n\geq K$, the set ${\mathcal L}_{n+1}(\mathbf u)$ can be partitioned into
\begin{enumerate}
\item ${\mathcal C}(k(n)+1)$ groups of $\Delta(n)$ words each, every word in the $j$th group having the frequency $\frac{1}{b^{p(n)}}\rho\bigl(v_{j}^{k(n)+1}\bigr)$, $j \in \{1,\dots, {\mathcal C}(k(n)+1)\}$,
\item ${\mathcal C}(k(n))$ groups of $b^{p(n)}-\Delta(n)$ words each, every word in the $j$th group having the frequency $\frac{1}{b^{p(n)}}\rho\bigl(v_{j}^{k(n)}\bigr)$, $j \in \{1,\dots, {\mathcal C}(k(n))\}$.
\end{enumerate}
\end{thm}
The frequencies $\rho\bigl(v_{j}^{(k)}\bigr), \ k\in \{K, \dots, b(K-1)+1\}$, can be found directly using~\eqref{ancestors}.
Theorem~\ref{explicit} provides then explicit formulae for factor frequencies of circular fixed points of uniform marked primitive morphisms.
%%%%%%%%%%%%%%%%%%%%%%%%%%%%%%%%%%%%%%%%%%%%%%%%%%%%%%%%%%%%%%%%%%%%%%%%%%%%%%%%%%%%%%%%%%%%%%%%%%%%%%
%%%%%%%%%%%%%%%%%%%%%%%%%%%%%%%%%%%%%%%%%%%%%%%%%%%%%%%%%%%%%%%%%%%%%%%
\subsection{Reduced Rauzy graphs}
Assume throughout this section that factor frequencies of infinite words in question exist.
The {\em Rauzy graph} of order $n$ of an infinite word $\mathbf u$ is
a~directed graph  $\Gamma_n$ whose set of vertices is $\L_n(\mathbf u)$ and  set of edges is $\L_{n+1}(\mathbf u)$. An edge $e = w_0
w_1 \dots w_n$ starts in the vertex $w=w_0w_1\dots w_{n-1}$, ends in
the vertex  $v=w_1\dots w_{n-1}w_n$, and is labeled by its factor
frequency $\rho(e)$.

It is easy to see that edge frequencies in a~Rauzy graph $\Gamma_n$ behave similarly as
the current in a~circuit. We may formulate an analogy of the Kirchhoff's current law:
the sum of frequencies of edges ending in a~vertex equals the sum of
frequencies of edges starting in this vertex.
\begin{obs}[Kirchhoff's law for frequencies]\label{KLaw}
Let $w$ be a~factor of an infinite word $\mathbf u$ whose factor frequencies exist. Then
$$\rho(w)=\sum_{a \in {\rm Lext}(w)}\rho(aw)=\sum_{a \in {\rm Rext}(w)}\rho(wa).$$
\end{obs}
The Kirchhoff's law for frequencies has some useful consequences.
\begin{coro}\label{KLaw1}
Let $w$ be a~factor of an infinite word $\mathbf u$ whose frequency exists.
\begin{itemize}
\item
If $w$ has a~unique right extension $a$, then $\rho(w)=\rho(wa)$.
\item
If $w$ has a~unique left extension $a$, then $\rho(w)=\rho(aw)$.
\end{itemize}
\end{coro}
\begin{coro}\label{KLaw2}
Let $w$ be a~factor of an aperiodic recurrent infinite word $\mathbf u$ whose frequency exists.
Let $v$ be the shortest BS factor containing $w$, then $\rho(w)=\rho(v)$.
\end{coro}
The assumption of recurrence and aperiodicity in Corollary~\ref{KLaw2} is needed in order
to ensure that every factor can be extended to a~BS factor.

Corollary~\ref{KLaw1} implies that if a~Rauzy graph contains a~vertex $w$ with only one
incoming edge $aw$ and one outgoing edge $wb$, then $\rho:=\rho(aw)=\rho(w)=\rho(wb)=\rho(awb)$.
Therefore, we can replace this triplet (edge-vertex-edge) with only one edge $awb$ keeping the frequency $\rho$.
If we reduce the Rauzy graph step by step applying the above
described procedure, we obtain the so-called {\em reduced Rauzy
graph} $\tilde{\Gamma}_n$, which simplifies the investigation of
edge frequencies. In order to precise this construction, we
introduce the notion of a~simple path.
\begin{de}\label{simple_path}
Let $\Gamma_n$ be the Rauzy graph of order $n$
of an infinite word $\mathbf u$ whose factor frequencies exist. A~factor $e$ of length larger than $n$ such that its prefix and its suffix of length $n$ are special factors and $e$ does not contain any other special factors is called a~simple path. We define the label of a~simple path $e$ as $\rho(e)$.
\end{de}
\begin{de}\label{reduced_Rauzy_graph}
The reduced Rauzy graph
$\tilde{\Gamma}_n$ of $\mathbf u$ of order $n$ is a~directed graph whose
set of vertices is formed by LS and RS factors of $\L_n(\mathbf u)$
and whose set of edges is given in the following way. Vertices $w$
and $v$ are connected with an edge $e$ if there exists in $\Gamma_n$
a~simple path starting in $w$ and ending in $v$. We assign to such
an edge $e$ the label of the corresponding simple path.
  \end{de}
\begin{pozn}
According to Corollary~\ref{KLaw1} and Definition~\ref{reduced_Rauzy_graph}, if $\mathbf u$ is an aperiodic recurrent infinite word 
whose factor frequencies exist, it holds for every $n \in \mathbb N$, 
$$\{\rho(e)\bigm | e \in {\mathcal L}_{n+1}({\mathbf u})\}=\{\rho(e) \bigm | e \ \text{edge in $\tilde{\Gamma}_n$}\}.$$ 
\end{pozn}
Considering Corollary~\ref{KLaw2} and Definition~\ref{reduced_Rauzy_graph}, one may observe the following.
\begin{obs}\label{StaciGrafyBS}
Let $\mathbf u$ be an aperiodic recurrent infinite word whose factor frequencies exist.
If we find to a~reduced Rauzy graph $\tilde{\Gamma}_n$ none of whose vertices is a~BS factor the reduced Rauzy graph of minimal larger order, say $\tilde{\Gamma}_m$, containing a~vertex being a~BS factor, then
    $$\{\rho(e) \mid e \ \text{edge in}\ \tilde{\Gamma}_n\}=\{\rho(e) \mid e \ \text{edge in}\ \tilde{\Gamma}_m\} \cup \{\rho(v) \mid v \ \text{vertex in} \  \tilde{\Gamma}_m\}.$$
\end{obs}
%\begin{obs}\label{SlabeBS}
%If $w$ is a~BS factor of $\mathbf u$ such that for every $a \in {\rm Lext}(w)$, there exists a~unique $b \in {\rm Rext}(w)$ satisfying $awb \in {\mathcal L}(\mathbf u)$ (let us call such BS factors {\em weak}), then $\rho(w)=\rho(aw)=\rho(wb)$. It means that in order to get $\{\rho(e) \mid e \ \text{edge in}\ \tilde{\Gamma}_n\}$, where all vertices of $\tilde{\Gamma}_n$ being BS factors are weak, then $\{\rho(e) \mid e \ \text{edge in}\ \tilde{\Gamma}_n\}=\{\rho(e) \mid e \ \text{edge in}\ \tilde{\Gamma}_{n+1}\}.$
%\end{obs}
%%%%%%%%%%%%%%%%%%%%%%%%%%%%%%%%%%%%%%%%%%%%%%%%%%%%%%%%%%%%%%%%%%%%%%%%%%%%%%%%%%%%%%%%%%%%%%%%%%%%%%
%%%%%%%%%%%%%%%%%%%%%%%%%%%%%%%%%%%%%%%%%%%%%%%%%%%%%%%%%%%%%%%%%%%%%%%
\section{Symmetries preserving factor frequency}
We will be interested in symmetries preserving in a~certain way factor occurrences in $\mathbf u$ and consequently, frequencies of factors of $\mathbf u$.
Let us call a~{\em symmetry} on ${\mathcal A}^*$ any mapping $\Psi$ satisfying the following two properties:
\begin{enumerate}
\item $\Psi$ is a~bijection: ${\mathcal A}^{*}\to {\mathcal A}^{*}$,
\item for all $w,v \in {\mathcal A}^{*}$
$$\#\{\text{occurrences of $w$ in $v$}\}=\#\{\text{occurrences of $\Psi(w)$ in $\Psi(v)$}\}.$$
\end{enumerate}
The following statements are taken from~\cite{Ba}.
\begin{thm}\label{reflection_letterpermutation}
Let $\Psi: {\mathcal A}^*\to {\mathcal A}^*$. Then $\Psi$ is a~symmetry if and only if $\Psi$ is a~morphism or an antimorphism such that $\Psi$ is a~letter permutation when restricted to $\mathcal A$.
\end{thm}
\begin{obs}\label{same_frequency}
Let $\mathbf u$ be an infinite word whose language is invariant under a~symmetry $\Psi$.
For every $w$ in $\L(\mathbf u)$ whose frequency exists, it holds
$$\rho(w)=\rho(\Psi(w)).$$
\end{obs}
We denote the set of all morphisms and antimorphisms on ${\mathcal A}^*$ by $AM(\mathcal A^*)$.
\begin{thm}\label{UpperBoundMoreSymmetries} Let $G \subset AM({\mathcal A}^*)$ be a~finite group containing an antimorphism and let $\mathbf u$ be a~uniformly recurrent aperiodic infinite word whose language is invariant under all elements of $G$ and such that the frequency $\rho(w)$ exists for every factor $w \in \L(\mathbf u)$.
Then there exists $M \in \mathbb N$ such that
\[\# \{\rho(e)\mid  e
\in \L_{n+1}(\mathbf u) \}\quad \leq \quad \frac{1}{\#G}\Bigl(4\bigl(\C(n+1)-\C(n)\bigr)+\#G- X-Y \Bigr) \quad \quad \text{for all $n \geq M$},\]
where $X$ is the number of BS factors of length $n$ and $Y$ is
the number of BS factors of length $n$ that are $\theta$-palindromes for an antimorphism $\theta \in G$.
\end{thm}
%%%%%%%%%%%%%%%%%%%%%%%%%%%%%%%%%%%%%%%%%%%%%%%%%%%%%%%%%%%%%%%%%%%%%%%%%%%%%%%%%%%%%%%%%%%%%%%%%%%%%%
%%%%%%%%%%%%%%%%%%%%%%%%%%%%%%%%%%%%%%%%%%%%%%%%%%%%%%%%%%%%%%%%%%%%%%%
\section{Factor frequencies of generalized Thue-Morse words}\label{ChFrequencyGeneralizedTM}
The generalized Thue-Morse word ${\mathbf t}_{b,m}$ is defined for $b \geq 2, \ m\geq 1, \ b,m \in \mathbb N$, as the
fixed point starting in $0$ of the morphism
\begin{equation}\label{TMmorphism}\varphi_{b,m}: k \ \to \ k(k+1)\dots(k+b-1),
\end{equation} where $k \in {\mathbb Z}_m=\{0,1,\dots, m-1\}$ and where the letters are expressed modulo $m$.
We denote by $q$ the smallest positive integer satisfying $q(b-1)=0\mod m$. The word ${\mathbf t}_{b,m}$ is periodic if and only if $b=1\mod m$ (see~\cite{AlSh}).
In this case, ${\mathbf t}_{b,m}=(01\dots (m-1))^{\omega}$, where $\omega$ denotes an infinite repetition. It is thus readily seen that any factor of ${\mathbf t}_{b,m}$ has its frequency equal to $\frac{1}{m}$.

\noindent Properties of $\varphi_{b,m}$ and ${\mathbf t}_{b,m}$:
\begin{enumerate}
\item $\varphi_{b,m}$ is primitive, therefore letter frequencies exist and are equal to the components of the eigenvector $\frac{1}{m}(1,1,\dots,1)^T$ of the incidence matrix corresponding to the dominant eigenvalue $b$,
\item $\varphi_{b,m}$ is uniform ($|\varphi_{b,m}(k)|=b$ for all $k \in {\mathbb Z}_m$),
\item $\varphi_{b,m}$ is marked,
\item ${\mathbf t}_{b,m}$ is circular with synchronization delay $L=2b$
\begin{proof}
 Any $w \in {\mathcal L}({\mathbf t}_{b,m})$ of length greater than or equal to $2b$ contains either for some $k, \ell \in {\mathbb Z}_m$ a~factor $k\ell$, where $l \not =k+1 \mod m$, or is of length $2b$ and of the form $w=k(k+1)\dots (k+2b-1)$ for some $k \in {\mathbb Z}_m$.
 \begin{enumerate}
 \item In the first case, it is easy to see that $k$ marks the end of $w_1$ and $\ell$ the beginning of $w_2$ in the synchronization point $(w_1, w_2)$ of $w$.
    \item In the second case, $(w,\varepsilon)$ is a~synchronization point of $w=k(k+1)\dots (k+2b-1)$.
    \end{enumerate}
\end{proof}
\item ${\mathcal L}({\mathbf t}_{b,m})$ is invariant under $D_m= \{\Pi_x\bigm | x \in {\mathbb Z}_m\} \cup \{\Psi_x\bigm | x \in {\mathbb Z}_m\},$
    where $\Pi_x$ is a~morphism and $\Psi_x$ an antimorphism defined for all $k \in {\mathbb Z}_m$ by
    $$\begin{array}{rcl}
    \Pi_x(k)&=&x+k, \\ \Psi_x(k)&=&x-k.\end{array}$$
For the proof see~\cite{St}.
\end{enumerate}

The aim of this section is to describe $\{\rho(e)\bigm |e \in {\mathcal L}_{n+1}({\mathbf t}_{b,m})\}$
for all $n \in \mathbb N$. Theorem~\ref{explicit} gives explicit formulae for factor frequencies if
the frequencies of factors of length $n \in \{1,\dots, 2b+1\}$ (the structure ordering number $K=3$ for ${\mathbf t}_{b,m}$) are known.
There is even an easier way to get factor frequencies using symmetries of ${\mathcal L}({\mathbf t}_{b,m})$ and the description of BS factors from~\cite{St}.
\begin{prop}\label{BS_TM}
If $v$ is a~BS factor of ${\mathcal L}({\mathbf t}_{b,m})$ of length greater than or equal to $2b$, then there exists a~BS factor $w \in {\mathcal L}({\mathbf t}_{b,m})$ such that $v=\varphi_{b,m}(w)$. Moreover, $\rho(v)=\frac{\rho(w)}{b}$.
\end{prop}
\begin{proof}
The first part has been proved as Lemma~3 in~\cite{St}. The second part follows from Proposition~\ref{unique_interpretation}.
\end{proof}

\noindent{\bf Reduced Rauzy graph method (RRG method)}\\
We get the frequencies $\{\rho(e)\bigm |e \in {\mathcal L}_{n+1}({\mathbf t}_{b,m})\}$ for all $n \in \mathbb N$ in the following way.
\begin{enumerate}
\item We describe reduced Rauzy graphs of order $n$, where $1\leq n \leq 2b-1$, making use of the invariance of ${\mathcal L}({\mathbf t}_{b,m})$ under symmetries. We notice that all of them contain a~BS factor as a~vertex.
\item Proposition~\ref{BS_TM} says that every BS factor is of length $b^{\ell}n$, $\ell \in \mathbb N$, where $n \in \{1,\dots, 2b-1\}$.
It is not difficult to see that all reduced Rauzy graphs of order greater than or equal to $2b$ containing a~BS factor as their vertex are obtained by a~repeated application of $\varphi_{b,m}$ simultaneously to vertices and edges of reduced Rauzy graphs of order $n$, where $2\leq n \leq 2b-1$. By Proposition~\ref{BS_TM}, the reduced Rauzy graph of order $nb^{\ell}$ obtained when $\varphi_{b,m}$ is applied $\ell$ times to the reduced Rauzy graph of order $n$, where $2\leq n \leq 2b-1$, satisfies
$$\begin{array}{rcl}
\{\rho(e) \bigm | e \ \text{edge in} \ \tilde{\Gamma}_{nb^{\ell}}\}&=&\{\frac{1}{b^{\ell}}\rho(e) \bigm | e \ \text{edge in} \ \tilde{\Gamma}_{n}\}, \\
\{\rho(v) \bigm | v \ \text{BS vertex in} \ \tilde{\Gamma}_{nb^{\ell}}\}&=&\{\frac{1}{b^{\ell}}\rho(v) \bigm | v \ \text{BS vertex in} \ \tilde{\Gamma}_{n}\}.
\end{array}$$
\item Applying Observation~\ref{StaciGrafyBS}, we obtain
\begin{enumerate}
\item If $(n-1)b^{\ell}<N<nb^{\ell}$ for some $n \in \{2,\dots, 2b\}$, then $$\{\rho(e)\bigm |e \in {\mathcal L}_{N+1}({\mathbf t}_{b,m})\}=\{\frac{1}{b^{\ell}}\rho(e)\bigm |e \ \text{edge in} \ \tilde{\Gamma}_{n}\}\cup \{\frac{1}{b^{\ell}}\rho(v)\bigm |v \ \text{BS vertex in} \ \tilde{\Gamma}_{n}\}.$$
\item If $N=nb^{\ell}$ for some $n \in \{2,\dots, 2b-1\}$, then $$\{\rho(e)\bigm |e \in {\mathcal L}_{N+1}({\mathbf t}_{b,m})\}=\{\frac{1}{b^{\ell}}\rho(e)\bigm |e \ \text{edge in} \ \tilde{\Gamma}_{n}\}.$$
\end{enumerate}
\end{enumerate}
\begin{ex}
Let us illustrate the RRG method for the Thue-Morse word ${\mathbf t}_{2,2}$.
\begin{enumerate}
\item
\begin{figure}[!h]\label{RauzyIllustration}
\begin{center}
\resizebox{12 cm}{!}{\includegraphics{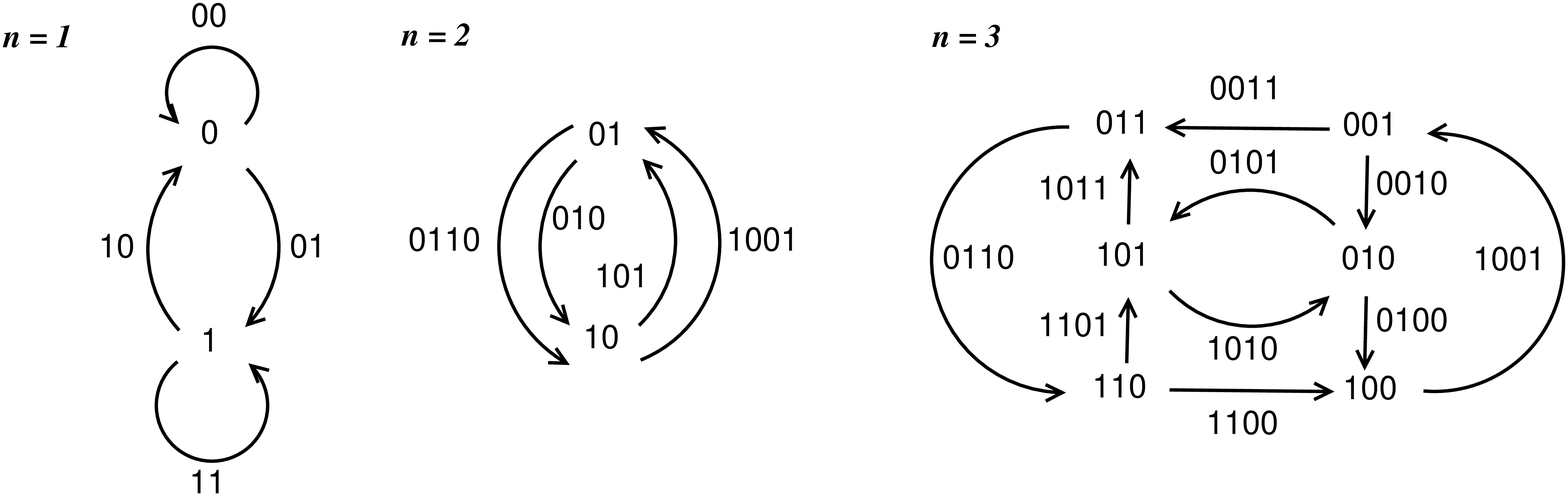}}
\end{center}
\caption{Reduced Rauzy graphs of ${\mathbf t}_{2,2}$ of order $n \in \{1,2,3\}$.}
\end{figure}
\noindent $\tilde{\Gamma}_1$: $\rho(0)=\rho(1)=\frac{1}{2}$ and $\{\rho(e)\bigm | e \ \text{edge in} \ \tilde{\Gamma}_1\}=\{\frac{1}{3}, \frac{1}{6}\}$.\\
\noindent Explanation:
\begin{itemize}
\item Thanks to Observation~\ref{same_frequency}, we have $\rho(0)=\rho(1)$, $\rho(01)=\rho(10)$, and $\rho(00)=\rho(11)$.
\item Using Property $(1)$ of $\varphi_{b,m}$ and ${\mathbf t}_{b,m}$, we get $\rho(0)=\frac{1}{2}$.
\item By Corollary~\ref{KLaw1} and Proposition~\ref{unique_interpretation}, it holds $\rho(00)=\rho(1001)=\frac{1}{2}\rho(10)$.
\item Applying the Kirchhoff's law for frequencies, we get $\rho(0)=\rho(01)+\rho(00)=\frac{3}{2}\rho(01)$, consequently $\rho(01)=\frac{1}{3}$.
\end{itemize}
\noindent $\tilde{\Gamma}_2$: $\rho(01)=\rho(10)=\frac{1}{3}$ and $\{\rho(e)\bigm | e \ \text{edge in} \ \tilde{\Gamma}_2\}=\{\frac{1}{6}\}$.\\
\noindent Explanation:
\begin{itemize}
\item Thanks to Observation~\ref{same_frequency}, we have $\rho(010)=\rho(101)$ and $\rho(0110)=\rho(1001)$.
\item By Proposition~\ref{unique_interpretation}, it holds $\rho(0110)=\frac{1}{2}\rho(01)=\frac{1}{6}$.
\item Applying the Kirchhoff's law for frequencies, we get $\rho(01)=\rho(010)+\rho(0110)$. Therefore $\rho(010)=\frac{1}{6}$.
\end{itemize}
\noindent $\tilde{\Gamma}_3$: $\rho(010)=\rho(101)=\frac{1}{6}$ and $\{\rho(e)\bigm | e \ \text{edge in} \ \tilde{\Gamma}_3\}=\{\frac{1}{6}, \frac{1}{12}\}$.\\
\noindent Explanation:
\begin{itemize}
\item Thanks to Observation~\ref{same_frequency}, we have $\rho(011)=\rho(100)=\rho(001)=\rho(110)$, $\rho(0011)=\rho(1100)$, $\rho(0101)=\rho(1010)$, $\rho(0010)=\rho(1101)=\rho(1011)=\rho(0100)$.
\item By Corollary~\ref{KLaw1} and Proposition~\ref{unique_interpretation}, it holds $\rho(0010)=\rho(100101)=\frac{1}{2}\rho(100)=\frac{1}{2}\rho(1001)=\frac{1}{12}$ and $\rho(0011)=\rho(100110)=\frac{1}{2}\rho(101)=\frac{1}{12}$.
\item The Kirchhoff's law for frequencies implies $\rho(0101)=\rho(010)-\rho(0100)=\frac{1}{12}$.
\end{itemize}
\item All reduced Rauzy graphs of order greater than or equal to $4$ containing a~BS factor as their vertex are depicted in Figure~\ref{RauzyIllustration2}.
\begin{figure}[!h]\label{RauzyIllustration2}
\begin{center}
\resizebox{14 cm}{!}{\includegraphics{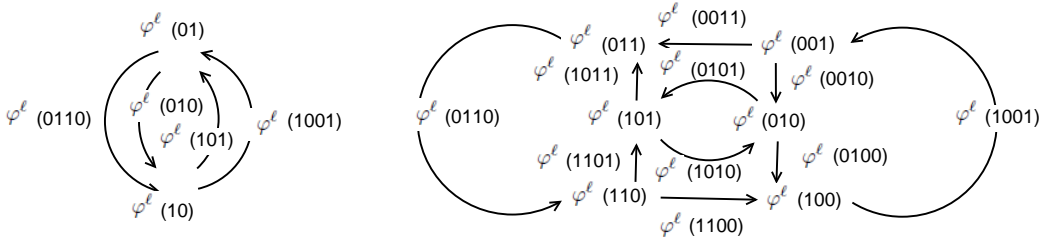}}
\end{center}
\caption{For any reduced Rauzy graph of ${\mathbf t}_{2,2}$ of order $n \geq 4$ containing a~BS factor as its vertex, there exists $\ell \geq 1$ such that the graph takes one of the depicted forms.}
\end{figure}
It holds for all $k \in \mathbb N$ $$\begin{array}{rcl}
\{\rho(e) \bigm | e \ \text{edge in} \ \tilde{\Gamma}_{2\cdot 2^{\ell}}\}&=&\{\frac{1}{2^{\ell}}\frac{1}{6}\}, \\
\{\rho(v) \bigm | v \ \text{BS vertex in} \ \tilde{\Gamma}_{2\cdot 2^{\ell}}\}&=&\{\frac{1}{2^{\ell}}\frac{1}{3}\}.
\end{array}$$
and $$\begin{array}{rcl}
\{\rho(e) \bigm | e \ \text{edge in} \ \tilde{\Gamma}_{3\cdot 2^{\ell}}\}&=&\{\frac{1}{2^{\ell}}\frac{1}{6}, \frac{1}{2^{\ell}}\frac{1}{12}\}, \\
\{\rho(v) \bigm | v \ \text{BS vertex in} \ \tilde{\Gamma}_{3\cdot 2^{\ell}}\}&=&\{\frac{1}{2^{\ell}}\frac{1}{6}\}.
\end{array}$$
\item The sets of factor frequencies $\{\rho(e)\bigm | e \in {\mathcal L}_{N+1}({\mathbf t}_{2,2})\}$  are of the following form for $N \in \mathbb N$.
\begin{enumerate}
\item $\{\rho(e)\bigm |e \in {\mathcal L}_{1}({\mathbf t}_{2,2})\}=\{\frac{1}{2}\}.$
\item $\{\rho(e)\bigm |e \in {\mathcal L}_{2}({\mathbf t}_{2,2})\}=\{\frac{1}{3}, \frac{1}{6}\}.$
\item If $2\cdot 2^{\ell}<N<3\cdot 2^{\ell}$ for some $\ell \in \mathbb N$, then $$\{\rho(e)\bigm |e \in {\mathcal L}_{N+1}({\mathbf t}_{2,2})\}=\left\{\frac{1}{2^{\ell}}\frac{1}{6}, \frac{1}{2^{\ell}}\frac{1}{12}\right\}\cup \left\{\frac{1}{2^{\ell}}\frac{1}{6}\right\}=\left\{\frac{1}{2^{\ell}}\frac{1}{6}, \frac{1}{2^{\ell}}\frac{1}{12}\right\}.$$
\item If $3\cdot 2^{\ell}<N<4\cdot 2^{\ell}$ for some $\ell \in \mathbb N$, then $$\{\rho(e)\bigm |e \in {\mathcal L}_{N+1}({\mathbf t}_{2,2})\}=\left\{\frac{1}{2^{\ell+1}}\frac{1}{6}\right\}\cup \left\{\frac{1}{2^{\ell+1}}\frac{1}{3}\right\}=\left\{\frac{1}{2^{\ell+1}}\frac{1}{3}, \frac{1}{2^{\ell+1}}\frac{1}{6}\right\}.$$
\item If $N=2\cdot 2^{\ell}$ for some $\ell \in \mathbb N$, then $$\{\rho(e)\bigm |e \in {\mathcal L}_{N+1}({\mathbf t}_{b,m})\}=\left\{\frac{1}{2^{\ell}}\frac{1}{6}\right\}.$$
\item If $N=3\cdot 2^{\ell}$ for some $\ell \in \mathbb N$, then $$\{\rho(e)\bigm |e \in {\mathcal L}_{N+1}({\mathbf t}_{b,m})\}=\left\{\frac{1}{2^{\ell}}\frac{1}{6}, \frac{1}{2^{\ell}}\frac{1}{12}\right\}.$$
\end{enumerate}
\end{enumerate}
\end{ex}
The RRG method says that it suffices to describe frequencies of edges and vertices being BS factors in reduced Rauzy graphs of order $n$, where $1\leq n \leq 2b-1$, in order to get $\{\rho(e) \bigm | e \in {\mathcal L}_{n+1}(e)\}$ for all $n \in \mathbb N$.
Using the description of BS factors from~\cite{St}, we obtain the form of reduced Rauzy graphs for $1 \leq n \leq 2b-1$.
\begin{enumerate}
\item For $1 \leq n \leq b$, the reduced Rauzy graph $\tilde \Gamma_n$ has $m$ vertices. All of them are BS factors of the form
$k(k+1)\dots (k+n-1)$. Since each of them is equal to $\Pi_k(01\dots(n-1)$, their frequencies are the same.
Moreover,
\begin{itemize}
\item $e$ is an edge ending in $01\dots (n-1)$ if and only if $\Pi_k(e)$ is an edge ending in $k(k+1)\dots(k+n-1)$,
\item $e$ is an edge ending in $01\dots (n-1)$ if and only if $\Psi_{k+n-1}(e)$ is an edge starting in $k(k+1)\dots \bigl(k+n-1)$,
\end{itemize}
and since $\rho(e)=\rho(\Pi_k(e))=\rho(\Psi_{k+n-1}(e))$, it suffices to describe frequencies of edges ending in $01\dots (n-1)$ in order to get all edge frequencies of $\tilde \Gamma_n$.
As shown in~\cite{St}, ${\rm Lext}(01\dots(n-1))=\{-1+k(b-1)\bigm | k \in \{0,1,\dots, q-1\}\}$.
\begin{lem}\label{freq_for_n_lessorequal_b}
Denote $f=\rho(01)$. Then $f=\frac{b^{q-1}}{m}\frac{b-1}{b^q-1}$ and for $1\leq n \leq b$, the frequencies of the vertex $w=01\dots (n-1)$ and of the edges ending in $w$ satisfy
$$
\begin{array}{rcll}
\rho(0)&=&\frac{1}{m},\\
\rho(01\dots(n-1))&=&(n-1)f-\frac{n-2}{m}& \text{for} \ n\geq 2,\\
\rho((-1)01\dots(n-1))&=&nf-\frac{n-1}{m},&\\
\rho((-1+k(b-1))01\dots (n-1))&=&\frac{1}{b^k}f& \text{for} \ k \in \{1,\dots, q-1\}.
\end{array}$$
\end{lem}
\begin{proof} Let us proceed by induction on $n$.
Let $n=1$, then $\rho(0)=\frac{1}{m}$ by Property $(1)$ of $\varphi_{b,m}$ and ${\mathbf t}_{b,m}$.
It holds by Corollary~\ref{KLaw1} and Proposition~\ref{unique_interpretation} for $k \in \{1,\dots, q-1\}$ that $$\rho((-1+k(b-1))0)=\rho(\varphi(-1+(k-1)(b-1))0))=\frac{1}{b}\rho((-1+(k-1)(b-1))0).$$
Thus $\rho((-1+k(b-1))0)=\frac{1}{b^k}\rho((-1)0)=\frac{1}{b^k}\rho(\Pi_{-1}(01))=\frac{1}{b^k}f$.
Using Observation~\ref{KLaw}, we obtain
$f=\rho(0)-\sum_{k=1}^{q-1}\frac{f}{b^k}$. Therefore $$f=\frac{b^{q-1}}{m}\frac{b-1}{b^q-1}.$$

Let $1<n+1\leq b$. Assume
$$\begin{array}{rcll}
\rho(01\dots(n-1))&=&(n-1)f-\frac{n-2}{m}& \text{for} \ n\geq 2,\\
\rho((-1)01\dots(n-1))&=&nf-\frac{n-1}{m},&\\
\rho((-1+k(b-1))01\dots (n-1))&=&\frac{1}{b^k}f& \text{for} \ k \in \{1,\dots, q-1\}.
\end{array}$$
Then,
$\rho(01\dots n)=\rho(\Pi_{-1}(01\dots n))=\rho((-1)01\dots (n-1))=nf-\frac{n-1}{m}$.
Applying Corollary~\ref{KLaw1}, we get
$\rho((-1+k(b-1))01\dots n)=\rho((-1+k(b-1))01\dots (n-1))=\frac{1}{b^k}f$.
Using the Kirchhoff's law for frequencies (Observation~\ref{KLaw}), we have 
$\rho((-1)01\dots n)=nf-\frac{n-1}{m}-\sum_{k=1}^{q-1}\frac{f}{b^k}=(n+1)f-\frac{n}{m}$.
\end{proof}
\item For $b+1 \leq n \leq 2b-1$, the reduced Rauzy graph $\tilde \Gamma_n$ has $3m$ vertices: $m$ of them are BS factors of the form
$\Pi_k(01\dots (n-1))$, $k \in {\mathbb Z}_m$, $m$ of them are LS factors of the form $\Pi_k(01\dots(b-1)1\dots (n-b))$, $k \in {\mathbb Z}_m$, $m$ of them are RS factors obtained by applying $\Psi_0$ to LS factors. Since symmetries preserve frequencies, all BS factors have their frequency equal to $\rho(01\dots (n-1))$ and similarly, all LS and RS factors have their frequency equal to $\rho(01\dots(b-1)1\dots (n-b))$.
By analogous arguments as in part $(1)$, we deduce that it suffices to describe frequencies of edges ending in $01\dots (n-1)$ and in $01\dots(b-1)1\dots (n-b)$ and the frequency of the unique edge $01\dots (b-1)1\dots (n+1-b)$ starting in $01\dots(b-1)1\dots (n-b)$ in order to get all edge frequencies of $\tilde \Gamma_n$.
Again by~\cite{St},
$$\begin{array}{rcl}
{\rm Lext}(01\dots(n-1))&=&\{-1, b-2\},\\
{\rm Lext}(01\dots(b-1)1\dots(n-b))&=&\{-1+k(b-1)\bigm | k \in \{0,1,\dots, q-1\}\}.
\end{array}$$
\begin{lem}\label{freq_for_n_greaterthan_b}
Denote $f=\rho(01)$. Then for $b+1\leq n \leq 2b-1$, the frequencies
\begin{enumerate}
\item of the BS vertex $w=01\dots(n-1)$ and of the edges ending in $w$ satisfy
$$\begin{array}{rcl}
\rho(01\dots(n-1))&=&\frac{1}{b^{q-1}}f-\frac{(n-b-1)}{b^q}f,\\
\rho((-1)01\dots(n-1))&=&\frac{1}{b^{q-1}}f-\frac{(n-b)}{b^q}f,\\
\rho((b-2)01\dots (n-1))&=&\frac{1}{b^q}f,
\end{array}$$
\item of the edge $01\dots (b-1)1\dots (n+1-b)$ starting in the LS vertex $v=01\dots(b-1)1\dots (n-b)$ and of the edges ending in $v$ satisfy
$$\begin{array}{rcll}
\rho(01\dots(b-1)1\dots (n+1-b))&=&\frac{1}{b}f,&\\
\rho((-1)01\dots(b-1)1\dots (n-b))&=&\frac{1}{b^q}f,&\\
\rho((-1+(b-1))01\dots (b-1)1\dots (n-b))&=&\frac{1}{b}(2f-\frac{1}{m}),& \\
\rho((-1+k(b-1))01\dots (b-1)1\dots (n-b))&=&\frac{1}{b^k}f &\text{for} \ k \in \{2,\dots, q-1\}.
\end{array}$$
\end{enumerate}
\end{lem}

\begin{proof} Let us proceed by induction on $n$.\\
\noindent $(a)$ Let $n=b+1$, then using part $(1)$, we obtain $\rho(01\dots b)=\rho(\Pi_{-1}(01\dots b))=\rho((-1)01\dots (b-1))=bf-\frac{b-1}{m}=\frac{1}{b^{q-1}}f$.
By Corollary~\ref{KLaw1}, Proposition~\ref{unique_interpretation}, and Observation~\ref{same_frequency},
we have $\rho((b-2)01\dots b)=\rho(\varphi((-1)0b))=\rho(\varphi^2((1-b)b))=\frac{1}{b^2}\rho((1-b)b)=\frac{1}{b^2}\rho(\Pi_{m-b}((1-b)b))=
\frac{1}{b^2}\rho((-1+(q-2)(b-1))0)=\frac{1}{b^q}f$.
Finally, applying Observation~\ref{KLaw}, we get
$\rho((-1)0\dots b)=\frac{1}{b^{q-1}}f-\frac{1}{b^q}f$.

Let $b+1<n+1\leq 2b-1$. Assume
$$\begin{array}{rcl}
\rho(01\dots(n-1))&=&\frac{1}{b^{q-1}}f-\frac{(n-b-1)}{b^q}f,\\
\rho((-1)01\dots(n-1))&=&\frac{1}{b^{q-1}}f-\frac{(n-b)}{b^q}f,\\
\rho((b-2)01\dots (n-1))&=&\frac{1}{b^q}f.
\end{array}$$
Then,
$\rho(01\dots n)=\rho(\Pi_{-1}(01\dots n))=\rho((-1)01\dots (n-1))=\frac{1}{b^{q-1}}f-\frac{(n-b)}{b^q}f$.
Applying Corollary~\ref{KLaw1}, we get
$\rho((b-2)01\dots n)=\rho((b-2)01\dots (n-1))=\frac{1}{b^q}f$.
Using the Kirchhoff's law for frequencies (Observation~\ref{KLaw}), we have
$\rho((-1)01\dots n)=\frac{1}{b^{q-1}}f-\frac{(n-b)}{b^q}f-\frac{1}{b^q}f=\frac{1}{b^{q-1}}f-\frac{(n+1-b)}{b^q}f$.

\vspace{0.3cm}

\noindent $(b)$
Let $n=b+1$, then by Corollary~\ref{KLaw1} and Proposition~\ref{unique_interpretation}, it follows
$\rho(01\dots(b-1)12)=\rho(\varphi(01))=\frac{1}{b}f$.
Again, by Corollary~\ref{KLaw1} and Proposition~\ref{unique_interpretation}, it holds for $k \in \{2,\dots, q-1\}$ that $\rho((-1+k(b-1))01\dots(b-1)1)=\rho(\varphi((-1+(k-1)(b-1))01))=\frac{1}{b}\rho((-1+(k-1)(b-1))01)=\frac{1}{b^{k}}f,$
and for $k=1$, we have by the same arguments
$\rho((-1+(b-1))01\dots(b-1)1)=\rho(\varphi((-1)01))=\frac{1}{b}\rho((-1)01)=\frac{1}{b}(2f-\frac{1}{m}).$
Finally, by the Kirchhoff's law for frequencies (Observation~\ref{KLaw}), we derive
$\rho((-1)01\dots(b-1)1)=\frac{1}{b}f-\frac{1}{b}(2f-\frac{1}{m})-\sum_{k=2}^{q-1}\frac{1}{b^{k}}f=\frac{1}{b^q}f.$

Let $b+1<n+1\leq 2b-1$. Assume
$$\begin{array}{rcll}
\rho(01\dots(b-1)1\dots (n+1-b))&=&\frac{1}{b}f,&\\
\rho((-1)01\dots(b-1)1\dots (n-b))&=&\frac{1}{b^q}f,&\\
\rho((-1+(b-1))01\dots (b-1)1\dots (n-b))&=&\frac{1}{b}(2f-\frac{1}{m}),& \\
\rho((-1+k(b-1))01\dots (b-1)1\dots (n-b))&=&\frac{1}{b^k}f &\text{for} \ k \in \{2,\dots, q-1\}.
\end{array}$$
By Corollary~\ref{KLaw1}, we have
$\rho(01\dots(b-1)1\dots (n+2-b))=\rho(01\dots(b-1)1\dots (n+1-b))=\frac{1}{b}f$.
Again by Corollary~\ref{KLaw1}, we get for all $k \in \{2,\dots, q-1\}$,
$\rho((-1+k(b-1)01\dots (b-1)1\dots (n+1-b))=\rho((-1+k(b-1)01\dots (b-1)1\dots (n-b))=\frac{1}{b^k}f$, and analogously,
$\rho((-1+(b-1)01\dots (b-1)1\dots (n+1-b))=\rho((-1+(b-1)01\dots (b-1)1\dots (n-b))=\frac{1}{b}(2f-\frac{1}{m})$.
Using the Kirchhoff's law for frequencies (Observation~\ref{KLaw}), we have
$\rho((-1)01\dots(b-1)1\dots (n+1-b))=\frac{1}{b}f-\frac{1}{b}(2f-\frac{1}{m})-\sum_{k=2}^{q-1}\frac{1}{b^k}f=\frac{1}{b^q}f$.
\end{proof}
\end{enumerate}
\begin{thm}\label{main} Let $b\geq 2, m\geq 1, b,m \in \mathbb N$, and $b \not =1 \mod m$. Let ${\mathbf t}_{b,m}$ be the fixed point starting in $0$ of the morphism $\varphi_{b,m}$ defined in~\eqref{TMmorphism}. Then the sets of factor frequencies take the following form for $N \in \mathbb N$.

\begin{tabular}{|l|l|}
\hline
& \\
$N$ & $\{\rho(e)\bigm | e \in {\mathcal L}_{{\mathbf t}_{b,m}}(N+1)\}$\\
\hline
\hline
& \\
$0$ & $\frac{1}{m}$\\
\hline
&\\
$1$ & $\frac{f}{b^k}, \ \text{where} \ k\in\{0,\dots, q-1\}$\\
\hline
&\\
$(n-1)b^{\ell}<N<nb^{\ell}, \ \ell \in \mathbb N$, & $\frac{1}{b^{\ell}}\left((n-1)f-\frac{n-2}{m})\right), \frac{1}{b^{\ell}}\left(nf-\frac{n-1}{m}\right), \frac{1}{b^{\ell}}\left(\frac{1}{b^k}f\right)$,\\
$\text{where} \ n \in\{3,\dots, b\}$&$ \text{where} \ k\in\{1,\dots, q-1\}$\\
\hline
& \\
$(n-1)b^{\ell}<N<nb^{\ell}, \ \ell \in \mathbb N$, &$\frac{1}{b^{\ell}}\left(\frac{1}{b^{q-1}}f-\frac{(n-b-1)}{b^q}f\right), \frac{1}{b^{\ell}}\left(\frac{1}{b^{q-1}}f-\frac{(n-b)}{b^q}f\right), \frac{1}{b^{\ell+1}}(2f-\frac{1}{m}),\frac{1}{b^{\ell}}\left(\frac{1}{b^k}f\right)$,\\
$\text{where} \ n \in\{b+1,\dots, 2b-1\}$& $\text{where}\  k \in \{1,\dots, q\}$\\
\hline
& \\
$(2b-1)b^{\ell}<N<2b^{\ell+1}, \ \ell \in \mathbb N$ &$\frac{1}{b^{\ell+1}}\left(2f-\frac{1}{m}\right), \frac{1}{b^{\ell+1}}\left(\frac{1}{b^k}f\right)$,\\
& $\text{where}\  k \in \{0,\dots, q-1\}$\\
\hline
& \\
$nb^{\ell}, \ \ell \in \mathbb N$,& $\frac{1}{b^{\ell}}\left(nf-\frac{n-1}{m}\right), \frac{1}{b^{\ell}}\left(\frac{f}{b^k}\right)$,\\
$\text{where}\ n\in \{2,\dots,b\}$& $\text{where} \ k\in\{1,\dots, q-1\}$\\
\hline
& \\
$nb^{\ell}, \ell\in \mathbb N$,& $\frac{1}{b^{\ell}}\left(\frac{1}{b^{q-1}}f-\frac{(n-b)}{b^q}f\right),\frac{1}{b^{\ell+1}}\left(2f-\frac{1}{m}\right), \frac{1}{b^{\ell}}\left(\frac{1}{b^k}f\right)$,\\
$\text{where} \ n \in \{b+1,\dots, 2b-1\}$& $\text{where}\  k \in \{1,\dots, q\}$\\
\hline
\end{tabular}
\end{thm}
\begin{proof}
The statement is obtained when putting together Lemmas~\ref{freq_for_n_lessorequal_b} and~\ref{freq_for_n_greaterthan_b}
and step $(3)$ of the RRG method.
\end{proof}

%%%%%%%%%%%%%%%%%%%%%%%%%%%%%%%%%%%%%%%%%%%%%%%%%%%%%%%%%%%%%%%%%%%%%%%%%%%%%%%%%%%%%%%%%%%%%%%%%%%%%%
%%%%%%%%%%%%%%%%%%%%%%%%%%%%%%%%%%%%%%%%%%%%%%%%%%%%%%%%%%%%%%%%%%%%%%%
\section{Upper bound on frequencies}
In the last section, let us show and explain that the optimal upper bound on the number of factor frequencies in infinite words whose language is invariant under more symmetries, here recalled as Theorem~\ref{UpperBoundMoreSymmetries}, is not reached for large enough $n$ for any generalized Thue-Morse word ${\mathbf t}_{b,m}$ with $b\geq 2, m\geq 1, b,m \in \mathbb N$, and $b \not =1 \mod m$.
First of all, the upper bound cannot be attained for $q>2:$ since $q$ corresponds to the number of extensions of special factors, 
the estimate $\#\{w\in {\mathcal L}_n(\mathbf u) \mid   w\ RS\}  \leq   \C(n+1)-\C(n)  = \sum_{w \in  {\mathcal L}_n(\mathbf u)}\bigl (\# \rm{Rext}(w)-1 \bigr),$ used in the proof of Theorem~\ref{UpperBoundMoreSymmetries}, is too rough for $q>2$. 
Indeed, the upper bound is greater than or equal to $2q-2$ 
(we have used the result $\C(n+1)-\C(n)\geq qm-m$ from~\cite{St}), while applying Theorem~\ref{main}, we can see that the number of frequencies of factors of the same length is at most $q+3$.

Nevertheless, even in the case of $q=2$, if we take $n=(2b-1)b^{\ell}$ for any $l \in \mathbb N$,
then $\#\{\rho(e)\bigm | e \in {\mathcal L}_{{\mathbf t}_{b,m}}(n+1)\}=3=q+1$ by Theorem~\ref{main}.
By the description of factor complexity from~\cite{St}, we have $\C(n+1)-\C(n)=qm=2m$.
It follows from Properties of $\varphi_{b,m}$ and ${\mathbf t}_{b,m}$ summarized in Section~\ref{ChFrequencyGeneralizedTM} that
$\#G=2m$ and the number of BS factors of length $n$ is equal to $m$ and is the same as the number of BS factors being $\theta$-palindromes for some antimorphism $\Psi_x, \ x \in {\mathbb Z}_m$. Therefore, the upper bound from Theorem~\ref{UpperBoundMoreSymmetries} is equal to $
\frac{1}{2m}\Bigl(8m+2m- m-m \Bigr)=4.$ Hence, for any $M\in \mathbb N$, the equality in the upper bound from Theorem~\ref{UpperBoundMoreSymmetries} is not reached for all $n \geq M$.

Let us explain the reason. In the proof of Theorem~\ref{UpperBoundMoreSymmetries}, we have used the invariance of ${\mathcal L}({\mathbf t}_{b,m})$ under symmetries in order to obtain the upper bound on the number of factor frequencies. However, some factors may have the same frequency for another reason. We observe as a~direct consequence of Corollary~\ref{KLaw1} the following.
\begin{obs}\label{SlabeBS}
If $w$ is a~BS factor of an infinite word $\mathbf u$ such that for every $a \in {\rm Lext}(w)$, there exists a~unique $b \in {\rm Rext}(w)$ satisfying $awb \in {\mathcal L}(\mathbf u)$ (let us call such BS factors {\em weak}), then $\rho(aw)=\rho(awb)=\rho(wb)$.
\end{obs}
For $n=2b-1$, the BS factor of the form $w=01\dots (2b-2)$ is weak: $w$ can be extended in only two ways, as $(b-2)w(2b-1)$ and as $(-1)wb$.  Hence, $\rho((-1)w)=\rho(wb)$ even if these factors are not symmetric images of each other.
Similarly, the BS factor $v=\varphi_{b,m}^{\ell}(w)$ of length $n=(2b-1)b^{\ell}$ is weak and $\rho(av)=\rho(vb)$, where $a$ is the last letter of $\varphi_{b,m}^{\ell}(-1)$, i.e., $a=-1+\ell(b-1)$.
It holds again that $av$ and $vb$ are not symmetric images of one another.
%%%%%%%%%%%%%%%%%%%%%%%%%%%%%%%%%%%%%%%%%%%%%%%%%%%%%%%%%%%%%%%%%%%%%%%%%%%%%%%%%%%%%%%%%%%%%%%%%%%%%%
%%%%%%%%%%%%%%%%%%%%%%%%%%%%%%%%%%%%%%%%%%%%%%%%%%%%%%%%%%%%%%%%%%%%%%%
%%%%%%%%%%%%%%%%%%%%%%%%%%%%%%%%%%%%%%%%%%%%%%%%%%%%%%%%%%%%%%%%%%%%%%%
\section{Acknowledgements}
I acknowledge financial support by the Czech Science Foundation
grant 201/09/0584, by the grants MSM6840770039 and LC06002 of the
Ministry of Education, Youth, and Sports of the Czech Republic.
%%%%%%%%%%%%%%%%%%%%%%%%%%%%%%%%%%%%%%%%%%%%%%%%%%%%%%%%%%%%%%%%%%%%%%%%%%%%%%%%%%%%%%%%%%%%%%%%%%%%%%
%%%%%%%%%%%%%%%%%%%%%%%%%%%%%%%%%%%%%%%%%%%%%%%%%%%%%%%%%%%%%%%%%%%%%%%
%%%%%%%%%%%%%%%%%%%%%%%%%%%%%%%%%%%%%%%%%%%%%%%%%%%%%%%%%%%%%%%%%%%%%%%

\end{document}